\documentclass[11pt]{amsart}

\usepackage{pdfsync}
\usepackage{amscd}
\usepackage{amsmath}
\usepackage{amssymb}
\usepackage{amsthm}
\usepackage{epsf}
\usepackage{latexsym}
\usepackage{verbatim}
\usepackage[utf8]{inputenc}
\usepackage{pdfsync}
\input epsf.tex

\newcounter{thmcount}
\newtheorem{theorem}[thmcount]{Theorem}
\newtheorem{definition}{Definition}[section]
\newtheorem{proposition}{Proposition}[section]
\newtheorem{lemma}[proposition]{Lemma}
\newtheorem{corollary}[proposition]{Corollary}

\theoremstyle{definition}
\newtheorem*{remark}{Remark}

\newcommand{\mc}[1]{\mathcal{ #1 }}

\newcommand{\Z}{\mathbb{Z}}

\newcommand{\N}{\mathbb{N}}

\newcommand{\diam}{\operatorname{diam}}

\newcommand{\ve}{\varepsilon}

\begin{document}

\title{On Li-Yorke Measurable Sensitivity}

\author[Hallett]{Jared Hallett}
\address[Jared Hallett]{Williams College\\ Williamstown, MA 01267, USA}
\email{jdh4@williams.edu}

\author[Manuelli]{Lucas Manuelli}
\address[Lucas Manuelli]{Princeton University\\ Princeton, NJ 08544}
\email{manuelli@mit.edu}

\author[Silva]{Cesar E. Silva}
\address[Cesar Silva]{Department of Mathematics\\
     Williams College \\ Williamstown, MA 01267, USA}
\email{csilva@williams.edu}


\begin{abstract}
The notion of Li-Yorke sensitivity has been studied extensively  in the case of topological dynamical systems.  We introduce a measurable version of Li-Yorke sensitivity, for nonsingular (and measure-preserving) dynamical systems, and compare it with various mixing notions. It is known that in the case of nonsingular dynamical systems, conservative ergodic Cartesian square implies double ergodicity, which in turn implies weak mixing, but the converses do not hold in general, though they are all equivalent in the finite measure-preserving case. We show that for nonsingular systems, ergodic Cartesian square implies Li-Yorke measurable sensitivity, which in turn implies weak mixing.  As a consequence we obtain that, in the finite measure-preserving case, Li-Yorke measurable sensitivity is equivalent to weak mixing.  We also show that with respect to totally bounded metrics, double ergodicity implies Li-Yorke measurable sensitivity.
\end{abstract}
\maketitle

\section{Introduction}\label{introduction}

The notion of sensitive dependence for topological dynamical systems  has been studied by many authors,  see, for example, the works \cite{BBCDS, GW93, A-B96} and the references therein. Recently,  various notions of measurable sensitivity have been explored in ergodic theory, see for example \cite{ABC02, He04, Cadre05, J-S08, Huang, G-S08}.

In this paper we are interested in formulating a measurable version of the topological notion of Li-Yorke sensitivity for the case of nonsingular and measure-preserving dynamical systems. 

In Section \ref{sec:measurable_sensitivity}, we review some preliminary definitions and introduce the notion of Li-Yorke measurable sensitivity (also called Li-Yorke M-sensitivity) which is based on the topological notion of Li-Yorke sensitivity in \cite{Ak02} and prove that in the conservative ergodic case it implies W-sensitivity introduced in \cite{G-S08}.  In Section~\ref{S:mixLiYorke} we prove that if
the Cartesian square is conservative ergodic (a condition stronger than weak mixing 
in the nonsingular case \cite{ALW79}) then it is Li-Yorke M-sensitive. 
Section~\ref{sec:li_yorke_sensitivity_and_weak_mixing} shows that for conservative ergodic nonsingular systems, 
ergodic Cartesian square implies Li-Yorke M-sensitivity, which in turn implies weak mixing;
a consequence of this is that in the finite measure-preserving case Li-Yorke sensitivity is 
equivalent to weak mixing. Section~\ref{S:scrambled} studies scrambled sets. In Section~\ref{sec:weak_mixing_and_w_measurable_sensitivity} we consider conservative ergodic infinite measure-preserving transformations which are not W-measurably sensitive. The final sections study entropy and the existence of 
a W-measurably sensitive $\mu$-compatible metric for any conservative ergodic transformation.

\subsection{Acknowledgements}
This paper is based on research by the Ergodic Theory group of the 2011 SMALL summer research project at Williams College. Support for the project was provided by National Science Foundation REU Grant DMS - 0353634 and the Bronfman Science Center of Williams College.
We would like to thank Sergiy Kolyada for bringing \cite{Ak02} to our attention.
We thank Sasha Danilenko and Philippe Thieullen for conversations that improved the paper.
We are indebted to the referee for comments and in particular for suggesting some arguments 
in Section 8.

\section{Preliminary Definitions and Measurable Sensitivity} 
\label{sec:measurable_sensitivity}


A \emph{nonsingular dynamical system} $(X,\mc{S},\mu,T)$ is a standard Borel space $(X,\mc{S})$ with a $\sigma$-finite, nonatomic measure $\mu$ and a nonsingular endomorphism $T:X\to X$ (i.e., for all $A\in\mc{S}$, $T^{-1}(A)\in S$ and $\mu A=0$ if and only if $\mu(T^{-1}A)=0$).  We sometimes take $T$ to be measure-preserving or the measure space to be finite.  We sometimes suppress $\mc{S}$ and write $(X,\mu,T)$.  We say that $T$ is conservative if for all $A$ of positive measure there exists 
$n>0$ such that $\mu(T^{-n}A\cap A)>0$. A set $A$ is positively invariant if 
$A\subset T^{-1}(A)$ and invariant if $T^{-1}(A)=A$. If $T$ is conservative and
$A$ is positively invariant, then $A$ is invariant mod $\mu$ (henceforth all equalities will be interpreted as mod $\mu$). A transformation $T$ is ergodic if whenever $A$ is invariant, then
$\mu(A)=0$ or $\mu(X\setminus A)=0$. A nonsingular transformation $T$ is weakly mixing
if whenever $f$ is an $L^\infty$ function such that $f\circ T=z f$ for $z\in\mathbb C$, then
$f$ is constant a.e. If $T\times T$ is ergodic then $T$ is weakly mixing, but the converse, while true in the finite measure-preserving case, does not hold in general \cite{ALW79}.

We consider metrics $d$ on $X$.  We assume throughout that these are Borel measurable 
on $X\times X$ and bounded by $1$.  We say a metric $d$ on $X$ is \emph{$\mu$-compatible} if $\mu$ assigns positive measure to nonempty, open $d$-balls \cite{J-S08, G-S08}.  It follows from  \cite[Lemma 1.1]{J-S08} that the topology generated by $d$ is separable.  Thus  open sets are measurable as they are countable unions of balls.  The notion of measurable sensitivity was introduced in \cite{J-S08}.

\begin{definition} We say a nonsingular dynamical system $(X,\mu,T)$ is \emph{measurably sensitive} if for every isomorphic mod 0 dynamical system $(X_1,\mu_1,T_1)$ and $\mu_1$-compatible metric $d$ on $X_1$ there exists a $\delta>0$ such that for all $x\in X_1$ and $\ve>0$ there is an $n\in\N$ such that
$$\mu_1\{y\in B_\ve(x):d(T_1^n(x),T_1^n(y))>\delta\}>0. $$ \end{definition}

This definition was refined in \cite{G-S08}.

\begin{definition}
Let $(X,\mu,T)$ be a nonsingular dynamical system and $d$ a $\mu$-compatible metric on $X$.  We say the system is \emph{W-measurably sensitive} with respect to $d$ if there is a $\delta>0$ such that for each $x\in X$
$$\limsup_{n\to\infty}d(T^nx,T^ny)>\delta $$
for a.e. $y\in X$.  The system is \emph{W-measurably sensitive} if it is W-measurably sensitive with respect to each $\mu$-compatible metric $d$.
\end{definition}

Proposition 7.2 in \cite{G-S08} shows that for conservative ergodic transformations these two notions  are equivalent.   Requiring that the condition hold for each $x\in X$ and a.e. $y\in X$ is equivalent to requiring that it hold for a.e. pair $(x,y)\in X^2$ (Proposition A.1, \cite{G-S08}), a notion called pairwise sensitivity introduced in \cite{Cadre05}.   The following classification result  is proved in \cite[Theorem 7.1]{G-S08}.

\begin{theorem}\label{prop:wmsclass} Let $(X,\mu,T)$ be a conservative ergodic nonsingular dynamical system.  Then $T$ is W-measurably sensitive or $T$ is isomorphic mod $0$ to an invertible minimal uniformly rigid isometry on a Polish space.\end{theorem}
	
	The proof of this theorem shows that the isometric metric is $\mu$-compatible.  The following are two technical propositions from \cite{G-S08} that are needed for our work. The proofs follow those in the aforementioned paper and are included for completeness.

	\begin{proposition}
		\label{sup}
		Suppose $T$ is a nonsingular transformation. If for almost every pair $(x,y) \in X \times X$ there exists $n \geq 0$ such that $d(T^n x, T^n y) \geq \delta$, then for almost every pair $(x,y) \in X \times X$ we have $\lim \sup_{n \to \infty} d(T^n x, T^n y) \geq \delta$.
	\end{proposition}

	\begin{proof} 
		Let
		\begin{equation*}
			Z(N,x) = \left \{ y \in X : \exists n, d(T^n (T^N x),T^n y) \geq \delta \right \}.
		\end{equation*}
	\noindent	 Then by hypothesis there exists a full measure set $A$ such that $\mu(Z(N,x)^c) = 0$ for each $x \in A$  and for all $N \in \mathbb{N}$. Let
		 \begin{equation*}
			Y(N,x) = \left \{ y \in X : \exists n > N, d(T^n x,T^n y) \geq \delta \right \}.
		\end{equation*}
	\noindent Note that $Y(N,x) = T^{-N}(Z(N,x))$. Since $Z(N,x)$ has full measure and $T$ is nonsingular we see that $Y(N,x)$ also has full measure. Hence
		\begin{equation*}
			\bigcap_{N > 0}Y(N,x)
		\end{equation*}
also has full measure. But this says for each $x \in A$ and almost every $y \in X$ that $\lim \sup_{n \to \infty} d(T^n x, T^n y) \geq \delta$. 
	\end{proof}
	
The proof of the next proposition follows from the same arguments as above.

	\begin{proposition}
		Suppose $T$ is a nonsingular transformation. If for almost every pair $(x,y) \in X \times X$ there exists $n \geq 0$ such that $d(T^n x, T^n y) \leq \delta$,
			then for almost every pair $(x,y) \in X \times X$ we have $\lim \inf_{n \to \infty} d(T^n x, T^n y) \leq \delta$.
	\end{proposition}

	The notion of W-measurable sensitivity adapts the notion of sensitivity to initial conditions from topological dynamics to the measurable case.  In topological dynamics there is also the notion of Li-Yorke sensitivity, in which points are not only required to separate but also to come back together. We give the definitions of topological Li-Yorke sensitivity as in \cite{Ak02}. 
	
	We recall the following definition from topological dynamics.

\begin{definition}\label{def:prox}
	Let $(X,d,T)$ be a topological dynamical system. A pair $(x,y)$ is said to be \emph{proximal} if
	\begin{equation*}
		\liminf_{n \to \infty} d(T^n x, T^n y) = 0.
	\end{equation*}
\end{definition}

Akin and Kolyada introduced the following in \cite{Ak02}.

\begin{definition}\label{def:LiYorke}
	Let $(X,d,T)$ be a topological dynamical system. We call the system \emph{Li-Yorke sensitive} if there exists an $\ve > 0$ such that every $x \in X$ is a limit of points $y \in X$ such that the pair $(x,y)$ is proximal but whose orbits are at least $\ve$ apart at arbitrarily large times.
\end{definition}

In this paper we consider the measure-theoretic analogue of Li-Yorke sensitivity.

\begin{definition}\label{def:} 
	Let $(X,\mu,T)$ be a nonsingular dynamical system and $d$ a $\mu$-compatible metric on $X$.  We say that a pair $(x,y)$ is a \emph{Li-Yorke pair} if
$$\begin{array}{ll}
\liminf_{n\to \infty}  d(T^n x,T^n y) =0 \text{ and } \limsup_{n\to \infty}  d(T^n x,T^n y) >0.
\end{array}$$
We say $(X,\mu,T)$ is \emph{Li-Yorke measurably sensitive for the metric $d$} if the set of Li-Yorke pairs $(x,y) \in X \times X$ has full measure.  We say it is \emph{Li-Yorke measurably sensitive} (henceforth Li-Yorke M-sensitive) if it is Li-Yorke M-sensitive for all $\mu$-compatible metrics.	
\end{definition}

Note that atomic measures are never sensitive.  We now show that Li-Yorke M-sensitivity is a (measurable) isomorphism invariant.  We need the following lemma (\cite{G-S08}, Lemma 6.1).

\begin{lemma}\label{prop:extendmetric}
Let $(X,\mc{S})$ be a standard Borel space with nonatomic measure $\mu$.  Let $U\subset X$ be a Borel subset of full measure and $d_U$ a $\mu$-compatible metric defined on $U$.  Then $d_U$ can be extended to a $\mu$-compatible metric $d_X$ on $X$ such that $d_U$ and $d_X$ agree on a set of full measure contained in $U\times U$.
\end{lemma}

The fact that the set is contained in $U\times U$ is not in the statement but follows from the proof of the lemma in the original paper.  We now proceed in a way similar to Proposition 6.2 of the same paper.

\begin{proposition}\label{prop1:}
Suppose nonsingular dynamical system $(X,\mu,T)$ is Li-Yorke M-sensitive.  Then any isomorphic system $(Y,\nu,S)$ is also Li-Yorke M-sensitive.
\end{proposition}
\begin{proof}
Suppose $(Y,\nu,S)$ is not Li-Yorke M-sensitive.  Then there is a $\nu$-compatible metric $d_Y$  on $Y$ for which $(Y,\nu,S)$ is not Li-Yorke M-sensitive.  As the systems are isomorphic, there are Borel sets $U\subset X$ and $V\subset Y$ of full measure and a bijection $\phi:U\to V$ such that $\phi\circ T=S\circ\phi$.  Define a $\mu$-compatible metric $d_U$ on $U$ by $d_U(x,y)=d_Y(\phi(x),\phi(y))$.  Applying Lemma \ref{prop:extendmetric} extends $d_U$ to a $\mu$-compatible metric $d_X$ on $X$ which agrees with $d_U$ on a set $X_0\subset U\times U$ of full measure in $X\times X$.  By hypothesis $T$ is Li-Yorke M-sensitive, so the set $L\subset X^2$ of Li-Yorke pairs has full measure.  It follows that the set $A= \bigcap (T\times T)^{-n}(X_0\cap L)$ has full measure.  Note that if $(x,y)\in A$, it is Li-Yorke, and for all $n$ we have that $d_U(T^nx,T^ny)=d_X(T^nx,T^ny)$ and $\phi(T^nx),\phi(T^ny)\in V$ (since $T^nx,T^ny\in U$).  Now $\phi\times\phi(A)$ has full measure in $Y\times Y$, and for all $(\phi(x),\phi(y))\in \phi\times\phi(A)$ we have for all $n$ that
$$d_Y(S^n\phi(x),S^n\phi(y))=d_U(T^nx,T^ny)=d_X(T^nx,T^ny).$$
It follows that all pairs in $\phi\times\phi(A)$ are Li-Yorke for $d_Y$, a contradiction.
\end{proof}

Li-Yorke M-sensitivity is not a priori stronger than W-measurable sensitivity.  But this turns out to be the case when the system in question is conservative ergodic.

\begin{proposition}
	Let $(X,\mu,T)$ be a conservative ergodic and nonsingular dynamical system.  If it is Li-Yorke M-sensitive, then it is W-measurably sensitive. \end{proposition}
\begin{proof}
We show the contrapositive.  If $T$ is not W-measurably sensitive, then by Theorem \ref{prop:wmsclass} it is isomorphic mod 0 to an isometry.  But then the isomorphic system is both Li-Yorke M-sensitive and an isometry for a $\mu$-compatible metric, which is impossible.
\end{proof}

\section{Mixing Conditions Stronger than Li-Yorke Sensitivity}\label{S:mixLiYorke}

We study not just the $\liminf$ and $\limsup$ but also the set of distances to which $T$ separates pairs.  Define
$$\mc{N}=\{r:\text{for a.e. }(x,y)\in X^2\text{, }\exists\{n_k\}\text{ s.t. }d(T^{n_k}x,T^{n_k}y)\to r\}. $$
\noindent  Notice that $0\in\mc{N}$ for Li-Yorke M-sensitive systems.  Further, it is easily shown that $\mc{N}$ is closed.  We then have the following result.  Let $G$ be the set of values taken on by the metric $d$.

\begin{proposition}	\label{prop:value} 
	 Let $T\times T$ be conservative ergodic nonsingular and $d$ be a $\mu$-compatible metric on $(X,\mu)$.  Then for a.e $(x,y)\in X^2$, we have that $(T^nx,T^ny)$ is dense in the product topology, so that for any $r\in G$ there is an $\{n_k\}$ such that $d(T^{n_k}x,T^{n_k}y)\to r$. \end{proposition}
\begin{proof}
Since $(X,d)$ is separable it has a countable topological basis $\{B_i\}$ of balls.  As the metric is $\mu$-compatible, $B_i\times B_j$ has positive measure in the product space for each $i$ and $j$.  It then follows from the ergodicity of $T\times T$ that 
$$A=\bigcap_{i,j}\bigcup_n (T\times T)^{-n}(B_i\times B_j) $$
has full measure.  By construction $(T^nx,T^ny)$ is dense in the product topology for all $(x,y)\in A$.  Now for each $k$ one can choose $n_k$ such that $T^{n_k}x\in B(z_1,1/k)$ and $T^{n_k}y\in B(z_2,1/k)$ where $d(z_1,z_2)=r$.
\end{proof}

\begin{corollary}Under the hypotheses of Proposition (\ref{prop:value}) $\overline{G}\subset\mc{N}$.  In particular, $T$ is Li-Yorke M-sensitive and W-measurably sensitive. \end{corollary}
	
We now investigate another condition under which $T$ is Li-Yorke M-sensitive.  The ergodicity of $T\times T$ is a strong condition in general.  We can extend the results from the $T\times T$ ergodic case to similar results in the doubly ergodic (i.e, for all sets of positive measure $A$ and $B$ iterates of $A$ intersect both $A$ and $B$ in positive measure) case.  The main difference is the necessity of requiring the metric to have some degree of compactness.  Define $H$ to be the set of $r\in G$ such that for every ball $B$ there is a point $x\in B$ and $y\in X$ such that $d(x,y)=r$.  Note that $0\in H$. 

\begin{proposition}
	\label{prop:doubly_ergodic}
	Let $T$ be doubly ergodic and let $d$ be a $\mu$-compatible metric such that $(X,d)$ is compact.  For each $r\in H$, for all $x \in X$, for a.e. $y \in X$ there is a sequence $\{n_k\}$ such that $d(T^{n_k}x,T^{n_k}y)\to r$.
\end{proposition}
\begin{proof}
Suppose to the contrary.  Then there is some $r\in H$, $x\in X$ and positive measure set $B\subset X$ such that for no $y$ in $B$ is it the case that $d(T^{n_k}x,T^{n_k}y)\to r$ for any sequence $\{n_k\}$.  Then for each $y\in B$ there is an $N$ and an $\ve$ such that $|d(T^nx,T^ny)-r|\geq \ve$ for all $n\geq N$.  Thus we can find an $N$ and an $\ve$ for which this property holds on a positive measure subset of $B$ (which for simplicity we just relabel $B$).  Cover $X$ with a finite number of $\ve/2$-balls $B_1,\dots,B_k$.  As $T$ is doubly ergodic, it is $k$-ergodic (\cite{DE01}), so we can choose an $n$ so that $\mu(T^{-n}B_i\cap B)>0$ for each $i=1,\dots,n$.  We have that $T^nx$ is contained in some ball $B_j$, and by hypothesis this ball contains a point $x'$ such that for some $y'$ in the space $d(x',y')=r$.  Now $y'$ is contained in some ball $B_\ell$.  Consider any $y\in T^{-n}B_\ell\cap B$.  Then $r-\ve< d(T^nx,T^ny)< r+\ve$ for all $y\in B'$, a contradiction.
\end{proof}

\begin{lemma}\label{lem:denoiso}
	If $T$ is doubly ergodic then it does not admit a $\mu$-compatible metric for which it is an isometry.
\end{lemma}

\begin{proof}
	Suppose not and let $d$ be a $\mu$-compatible isometric metric. Choose an $\ve > 0$ such that there exists sets $A,B \subset X$ with $d(A,B) > \ve$. Since $d$ is $\mu$-compatible $\mu(B_{\ve/2}(x)) > 0$. By double ergodicity there exists an $n$ such that
	$$\begin{array}{cc}
		\mu(T^{-n}(B_{\ve/2}(x)) \cap A) > 0&\mu(T^{-n}(B_{\ve/2}(x)) \cap B) > 0.
\end{array}$$
	This implies that the diameter of $T^n(B_{\ve/2}(x))$ is strictly greater than $\ve$ contradicting the fact that $T$ is an isometry for $d$.
\end{proof}

\begin{corollary}
	\label{cor:totally_bounded}
Let $T$ be doubly ergodic and $d$ be a totally bounded $\mu$-compatible metric.  Then $T$ is Li-Yorke M-sensitive with respect to $d$.
	 \end{corollary}
\begin{proof}
By Proposition 6.3 \cite{G-S08} we know that $T$ is either W-measurably sensitive or is measurably isomorphic to an isometry. Since it is doubly ergodic the preceding lemma shows that it cannot be isomorphic to an isometry and hence is W-measurably sensitive. Thus $T$ separates points.  

We now show that $T$ brings points together.  Suppose it did not.  Proceeding as in the proof of Proposition \ref{prop:wmsclass} (Theorem 7.1 in \cite{G-S08}), let $(X_1,d_1)$ be the topological completion of $(X,d)$.  We can extend the measure to this space by letting sets $S\subset X_1$ be measurable when $S\cap X$ is measurable and by defining $\mu_1(S)=\mu(S\cap X)$.  Define $T_1$ on $X_1$ as $T$ on $X$ and as the identity on $X_1\backslash X$.  This new system is isomorphic mod 0 to $(X,\mu,T)$, so $T$ is doubly ergodic.  Since $(X,d)$ is totally bounded, $(X_1,d_1)$ is compact.  Then by Proposition \ref{prop:doubly_ergodic}, $T_1$ brings points together in the metric $d_1$.  It follows that $T$ brings points together in the metric $d$.
\end{proof}

\section{Li-Yorke M-sensitivity and Weak Mixing} 
\label{sec:li_yorke_sensitivity_and_weak_mixing}


We now show that Li-Yorke sensitivity implies weak mixing, which will be a converse to the above results in the finite measure-preserving case. We first show the useful fact that every positive measure subset of a standard Borel space admits a $\mu$-compatible metric, which in fact turns out to be totally bounded.  We start with a lemma whose proof is standard.

\begin{lemma}\label{lem:sigma_finite}
	Given a measure space $(X,\mathcal{S},\mu)$ with a $\sigma$-finite measure $\mu$ there exists an equivalent probability measure $\eta$.
\end{lemma}

\begin{proposition}
	\label{prop:compatible_metric}
	Let $(X,\mathcal{B},\mu )$ be a standard Borel space with nonatomic $\sigma$-finite measure $\mu$. Then there exists a totally bounded $\mu$-compatible metric $d$ on $X$.
\end{proposition}

\begin{proof}\label{pf1:}
	By Lemma \ref{lem:sigma_finite} there exists an equivalent probability measure $\eta$. Now using the isomorphism theorem for measures (\cite{Ke95}, Theorem $17.41$), there is a Borel isomorphism $f:X \to [0,1]$ with $f \eta = m$ where $m$ is the Lebesgue measure on $[0,1]$. Define a metric $d$ on $X$ by 
	\begin{equation*}
		d(x,y) = |f(x) - f(y)|.
	\end{equation*}
	Note that this defines a totally bounded metric. Now we show it is $\mu$-compatible. Given $x \in X$ and $\ve > 0$ consider $B_{\ve}(x)$. The set $f(B_{\ve}(x))$ is an open ball of radius $\ve$ about $f(x)$ on the interval $[0,1]$ and thus has positive Lebesgue measure. As $m(f(B_{\ve})(x)) > 0$, we have that $\eta(B_{\ve}(x)) > 0$, so that $\mu(B_{\ve}(x)) > 0$. Hence $d$ is $\mu$-compatible.
\end{proof}

\begin{corollary}\label{cor:compatible_metric}
	Let $(X,\mathcal{B},\mu )$ be a standard Borel space with nonatomic $\sigma$-finite measure $\mu$ and $Y \in \mathcal{B}$ be a positive measure subset of $X$. Then there exists a totally bounded $\mu$-compatible metric $d$ on $Y$.
\end{corollary}

\begin{proof}\label{pf2:}
	By  \cite[Corollary 13.4]{Ke95}) we know that $Y$ is   a standard Borel space and the restriction of $\mu$ to $Y$ is  a nonatomic nontrivial $\sigma$-finite measure.
	\end{proof}

We now show the main result of this section.

\begin{proposition}\label{prop:weak_mixing}Let $(X,\mu,T)$ be a conservative ergodic nonsingular dynamical system. If $T$ is Li-Yorke M-sensitive, then it is weakly mixing.

\end{proposition}
\begin{proof}
	
	We prove the contrapositive. Suppose $T$ is not weakly mixing, so that it has a nonconstant  eigenfunction $f:X\to S^1$.  This is a factor map mapping $(X,\mu,T)$ to $(S^1,\nu,R)$ where $\nu$ is some measure on the circle and $R$ is a rotation.  Given $z,z' \in S^1$ let $a(z,z')$ denote the normalized arc length on the circle (so $a(1,-1) = \frac{1}{2}$). Since the map $f$ is nontrivial and measure-preserving there exists an $m$ such that the set $A \subset X^2$ defined by	
	\begin{equation*}
		A = \left \{ (x,y) \in X \times X : a(f(x),f(y)) \geq 1/m\right \}
	\end{equation*}
\noindent	has positive measure. Let $D_i = f^{-1}([e^{\frac{2 \pi (i-1)}{m}},e^{\frac{2 \pi i}{m}}))$ for $i = 1,\ldots,m$. Relabel the $D_i$ as
	\begin{equation*}
		\{E_1,\ldots,E_k,\ldots,E_m\},
	\end{equation*}
	
\noindent	where $\mu(E_j) = 0$ for $j > k$. Then we have
	\begin{equation*}
		X = (E_1 \sqcup \ldots \sqcup E_k)\mod 0.
	\end{equation*}
	By removing from $X$ the backwards orbits of $E_j$ for $j > k$, a measure zero set, we obtain a $T$-invariant set of full measure. Thus since Li-Yorke M-sensitivity is preserved under measurable isomorphism we can assume that $X =E_1 \sqcup \ldots \sqcup E_k $. By Corollary \ref{cor:compatible_metric} we know that for each $j \leq k$ there exists a metric $d_j$ on $E_j$ which is $\mu$-compatible. Define a metric $d$ on $X$ by
	\begin{equation*}
	d(x,y) = \begin{cases}
			1 & \text{if } x \in E_i, y \in E_j, i \neq j\\
			d_i(x,y) & \text{if } x,y \in E_i.
		\end{cases}		
	\end{equation*}	
	Note that $d$ is a $\mu$-compatible metric for $X$.  As $A$ has positive measure in the product space, showing that no $(x,y)\in A$ is proximal suffices to show that $T$ is not Li-Yorke M-sensitive.  Note that since $R$ is a rotation it is an isometry under arc length. Then if $(x,y) \in A$,
	\begin{equation*}
		a(f(T^n x), f(T^n y)) = a(R^n f(x),R^n f(y)) = a(f(x),f(y)) \geq 1/m.
	\end{equation*}	
	This implies that $T^nx$ and $T^ny$ cannot lie in the same $E_i$, for otherwise we would have that $a(f(x),f(y)) < 1/m$. Thus $d(T^n x ,T^n y) = 1$ for all $n \geq 0$, which implies $(x,y)$ is not proximal.
\end{proof}

Collecting our results we see that we have the following chain of implications.

\begin{theorem}\label{thm:}
	Let $(X,\mu,T)$ be a conservative ergodic nonsingular dynamical system.  Then	
	\begin{equation*}
		T \times T \text{ ergodic } \Rightarrow T \text{ Li-Yorke M-sensitive } \Rightarrow T \text{ weakly mixing}.
	\end{equation*}
\end{theorem}

In the finite measure-preserving case it is known that weak mixing and ergodicity of the product are equivalent.  We thus have the following result.

\begin{theorem}\label{thm:weak_mixing}
	Suppose $(X,T,\mu)$ is a finite measure-preserving ergodic dynamical system. Then $T$ is Li-Yorke M-sensitive if and only if it is weakly mixing.
\end{theorem}

Note that in this case $T$ is not just Li-Yorke M-sensitive but in fact separates points arbitrarily close to every possible distance that any given $\mu$-compatible metric assumes.  We now show the following lemma, which will allow us to relate this result to the classification result on W-measurable sensitivity (Theorem \ref{prop:wmsclass}).

\begin{lemma}\label{lem1:}
	Suppose $(X,\mu,T)$ is an ergodic finite measure-preserving transformation. Then $T$ is weakly mixing if and only if $T$ has no nontrivial factors which admit a $\mu$-compatible metric for which they are an isometry.
\end{lemma}

\begin{proof}\label{pf3:}
	If $T$ is weak mixing, then all its factors are doubly ergodic and so by Lemma \ref{lem:denoiso} do not admit $\mu$-compatible isometric metrics.  If $T$ is not weak mixing, then it has a nontrivial rotation factor.  If the rotation is irrational, then the measure in the factor must be Lebesgue measure, for which arc length is an isometry.  Otherwise, the factor is a rotation on $n$ points.  Spacing these atoms evenly around the unit circle and using arc length gives an isometric metric.
\end{proof}

\begin{corollary}\label{cor:}
	Suppose $(X,\mu,T)$ is a finite measure-preserving ergodic dynamical system. Then $T$ is Li-Yorke M-sensitive if and only if no nontrivial factors of $T$ admit a $\mu$-compatible isometric metric.
\end{corollary}


\section{Scrambled Sets}\label{S:scrambled}
The notion of a scrambled set in the topological case goes back to Li and Yorke \cite{LiY75}. Akin and Kolyada \cite{Ak02}   take the existence of an uncountable scrambled set as indicating that a system is chaotic.  Here we show that Li-Yorke M-sensitive systems always contain an uncountable scrambled set.

\begin{definition}\label{def:scram}
	A set $A \subset X$ is said to be scrambled if for all $x,y \in A$ with $x \neq y$ the pair $(x,y)$ is Li-Yorke.
\end{definition}

\begin{proposition} Let $A\subset X^2$ have positive measure and suppose there exists $D \in X$ with positive measure such that $A \subset D \times D$ and $A$ has full measure in $D \times D$. Then for almost every $x \in D$ there exists an uncountable set $B\subset X$ containing $x$ such that $y,z\in B$ with $y\neq z$ implies $(y,z)\in A$.\end{proposition}
\begin{proof}
In this context we use the term scrambled set to refer to a set $C$ such that $y,z\in C$ and $y\neq z$ implies $(y,z)\in A$.  We can assume without loss of generality that $D^2 = X^2$ and that $A$ has full measure. Then a.e. $x\in X$ has a full-measure fiber, i.e. $\mu\{y\in X:(x,y)\in A\}=1$.  Let the set of such $x$ be denoted $X_0$.  For any $x,y\in X_0$ let $P$ be the collection of scrambled sets contained in $X_0$ and containing $\{x,y\}$.  This is nonempty as $\{x,y\}\in P$.  Let $P$ be partially ordered by set inclusion and take $Q$ to be any totally ordered subset.  Then $\bigcup_{S \in Q}S$ is contained in $P$ and is an upper bound for $Q$.  Thus every totally ordered subset of $P$ has an upper bound. Now Zorn's Lemma guarantees the existence of an element $M$ such that $M\in P$ and $M$ is not contained in any other element of $P$.  We show $M$ is uncountable.  Suppose it were not.  Then for $z\in A$ let $C(z)=\{w\in X:(z,w)\text{ is Li-Yorke}\}$.  We have $\mu(C(z)^c)=0$ so that $\mu((\bigcap_{z\in A}C(z))^c)=0$.  In particular $\bigcap_{z\in A}C(z)$ is nonempty and so contains some $w$.  But then $M'=M\cup\{w\}$ is scrambled and $M\subsetneq M'$.
\end{proof}

\begin{corollary}
	\label{scrambled}
	Suppose $T$ is a Li-Yorke M-sensitive nonsingular transformation. Then a.e. $x$ is contained in an uncountable scrambled set.
\end{corollary}

Thus a Li-Yorke M-sensitive transformation is also chaotic in the sense of \cite{Ak02} when viewed as a topological dynamical system under any $\mu$-compatible metric.  Uncountable scrambled sets play a smaller role here than they do in topological dynamics because uncountable sets can still be measure-theoretically small.  However, this idea will be important in proving that periodic transformations admit no $\mu$-compatible sensitive metrics.

\section{Weak Mixing and W-measurable sensitivity} 
\label{sec:weak_mixing_and_w_measurable_sensitivity}

Weak mixing implies W-measurable sensitivity in the finite measure-preserving case, as in this case weak mixing is equivalent to double ergodicity.  We do not know if this is true in the general nonsingular case. Here we study in more detail when they system is not W-measurably sensitive.  
 We start with the following lemma that proves a certain uniformity of the measures for $\mu$-compatible
 isometries.

\begin{lemma}\label{lem2:}
	Let $(X,\mu,T)$ be a conservative ergodic measure-preserving dynamical system and $d$ a $\mu$-compatible metric on $X$ for which $T$ is an invertible isometry.  Then the measure of balls $B_r(x)$ depends only on the radius.
\end{lemma}

\begin{proof}\label{pf4:}
By Lemma 5.4 in \cite{G-S08}, almost every point of $X$ is transitive.  Let $x$ be such a point, and consider $f(r)=\mu B(x,r)$.  As $\mu$ is nonatomic $f(0)=0$ and since the function is increasing we may define $R=\sup\{r: \mu B(x,r)<\infty\}$.  Then $f$ is continuous for $r < R$ and moreover left-continuous at $R$ if $R<\infty$.
Let $y\in X$ and $r,\ve>0$.  There is an $n$ such that $d(T^nx,y)<\ve$. Then
 $T^n(B_{r-\ve}(x)) \subset B_r(y) \subset T^n(B_{r+\ve}(x)),$ so that \[\mu (B(x,r-\ve))\leq \mu (B(y,r))\leq \mu (B(x,r+\ve)).\]
For $r<R$ we have then by continuity that $\mu B(x,r)=\mu B(y,r)$.

If $r=R<\infty$, then applying left-continuity and the $r<R$ case we find $\mu B(y,R)=\lim_{\ve\to 0^+} \mu B(y,R-\ve)=\lim_{\ve\to 0^+} \mu B(x,R-\ve)=\mu B(x,R)$.
If $r>R$, then for some $\ve$ we have $\mu (B(y,r)) \geq \mu (B(x,r-\ve))=\infty$, hence $\mu (B(y,r))=\infty=\mu (B(x,r))$.
\end{proof}

The following  proposition makes explicit  some conclusions of the classification theorem in \cite{G-S08} (Theorem~\ref{prop:wmsclass}), and introduces notation that we use when we prove in the remark below that locally finite measure metrics are not possible when $T$ is infinite measure-preserving. These are basically additional conclusions in 
 Proposition 5.6 of \cite{G-S08} that were not stated there explicitly (namely, that every element of $G$ is an isometry and the group is abelian); as mentioned in
 \cite{G-S08},  Proposition 5.6,  is essentially from Akin and Glasner \cite{A-G01}. 

\begin{proposition}[\cite{G-S08}] \label{prop:imp}
	Let $(X,\mu,T)$ be a conservative ergodic nonsingular dynamical system.  If $T$ is not W-measurably sensitive, then it is isomorphic mod 0 to a minimal uniformly rigid invertible isometry on a monothetic, hence abelian, Polish group of isometries.
\end{proposition}

\begin{proof} Suppose $T$ is not W-measurably sensitive with respect to a $\mu$-compatible metric $d$.
As then $d$ is   separable,    there is a dense set of points $\{x_i\}$,  thus collection of balls of rational radius about each $x_i$ is a countable collection of Borel sets that separates points; by Blackwell's theorem \cite[4.5.10]{Sr1998} it generates the Borel sets of $X$.   Lemma 5.3 in \cite{G-S08} shows that there is a positively invariant Borel  set of full measure $X_0$ and
	a  $\mu$-compatible metric $d_T$ on $X_0$ for which $T$ is an isometry, where 
	$d_{T}(x,y)=\sup_{n\geq0}d(T^{n}x,T^{n}y)$. 
	In Theorem~\ref{prop:wmsclass} \cite{G-S08}, $(X_0,\mu,d_T,T)$ is
	extended to its topological completion (of which it is a Borel subset as $X_0$ is standard, \cite[p. 134]{Par67}). Relabel this system $(X,\mu,T)$ with metric $d_X$.  Proposition 5.6 in \cite{G-S08} and the discussion preceding it demonstrate the following.  Let $G$ be the set of continuous transformations on $X$ which commute with $T$.  Then $G$ is a group.  For $S_1,S_2\in G$ define $d_G(S_1,S_2)=\sup_{x\in X}d_X(S_1x,S_2x)$.  Then $d_G$ is a metric on $G$ and $G$ is equal to the closure of the set $\{\mathbb{I},T,T^2,\dots\}$, so it is a monothetic group.  Moreover, for any $x\in X$ the map $\phi_x:G\to X$ given by $\phi_x(S)=Sx$ is an isometry.  Choose one such map $\phi$ and use it to induce a measure $\eta$ on $G$.  Let $T:G\to G$ denote the group rotation on $G$ by the element $T$.  Then $(G,\eta,T)$ is isomorphic to $(X,\mu,T)$ by $\phi$.   
	
Since $\phi_x$ is an isometry for all $x$, $d_X(S_1x,S_2x)=d_G(S_1,S_2)= 
d_X(S_1y,S_2y)$ for all $x,y\in X$. Then a right cancellation law holds:
$d_G(S_1S,S_2S)=d_G(S_1Sx,S_2Sx)=d_G(S_1(Sx),S_2(Sx))=d_G(S_1,S_2)$, 
for $S,S_1,S_2\in G$. For every map $S\in G$ there is a sequence $T^{n_k}$ which converges uniformly to $S$.  It follows that $S$ is an isometry on $(G,d_G)$. That $G$ is abelian is a direct consequence since it is monothetic:  Let $S_1,S_2\in G$ and $\ve>0$.  Choose $k$ so that $d_G(T^kS_1,S_2)<\ve/2$.  Then
$d_G(S_1S_2,S_2S_1)\leq d_G(T^kS_1S_2,S_2S_2)+d_G(S_2S_2,S_2T^kS_1)=2d_G(T^kS_1,S_2)<\ve.$
\end{proof}

We will need the following lemma, that may be standard, to show local compactness.
Say a metric $d$ is {\it locally measure finite} if  each point of the space is
contained is an open $d$-ball of finite measure.

\begin{lemma}\label{L:locfincompact}  Let $(X,d)$ be a complete metric space. Let $\mu$ be a Borel measure on $X$
such that all open balls of the same  radius have the same measure, and such that $d$ is locally
measure finite.  Then $(X,d)$ is locally compact.
\end{lemma}

\begin{proof}
For  $x$ in $X$ there exists $r>0$ so that for all $0<\varepsilon <r/2$ the closure of
the ball $B(x,r+\varepsilon)$, denoted $E$, has finite measure.  As all balls have
the same radius we may assume $r$ works for all $x$.  Let $0<\varepsilon <r/2$ and $K$ be the set of all positive integers $k$
so that there are $k$  disjoint balls each of radius $\varepsilon$ contained in $E$. As all balls have the same measure, $K$ is finite. Let $\ell$ be its maximum and $B(y_1,\varepsilon),\ldots, B(y_\ell,\varepsilon)$ be balls that attain this maximum. Then $E\setminus \bigcup_{i=1}^\ell B(y_i,\varepsilon)$ does not contain
any ball of radius $\epsilon$. It follows that $ \bigcup_{i=1}^\ell B(y_i,2\ve)$
covers the closure of $B(x,r)$, denoted $E_0$. For if not there would exist $y$ in $E_0$ such
that $d(y,y_i)\geq 2\ve$ for all $i=1,\ldots,\ell$, which would imply that 
$B(y,\ve)$ is in $E\setminus \bigcup_{i=1}^\ell B(y_i,\varepsilon)$, a contradiction.  So  $E_0$ is totally bounded, and since it is closed and $X$ is  complete, it is compact. 
\end{proof}

\begin{remark} Let $T$ be conservative ergodic and  infinite measure-preserving on $(X,\mu)$.  
Suppose $T$ is not W-measurably sensitive with respect to a $\mu$-compatible metric $d$ and $d$
is locally measure finite. Using the notation in the proof of  Proposition~\ref{prop:imp} (and Theorem~\ref{prop:wmsclass} \cite{G-S08}), the metric $d_T$  is also locally measure finite
(since $B^{d_T}(x,r)\subset B^d(x,r)$), and so is its completion.
Relabel this system $(X,\mu,T)$ with metric $d_X$ and 
let $(G,\eta,T)$  and $d$ also be as in the proof of Proposition~\ref{prop:imp}.
As $\phi$ is an isometry and $\eta$ is the image measure, $d$ on $(G,\eta)$ is also locally measure finite. By Lemma~\ref{lem2:}, balls of the same radius in  
$(G,\eta,T,d)$ have the same measure, so by  Lemma~\ref{L:locfincompact}  $(G,d) $ is locally compact. As $T$ is minimal, $\Z$  is dense in the locally compact group $G$, and since $G$ is not discrete, then $G$ would be compact, see, for example
  \cite[p. 71]{Morr77} (this is just a folk theorem that there is no minimal map on a noncompact, locally compact space). The local measure finiteness property implies that $\eta$ is a finite measure,
   contradicting that $T$ is infinite measure-preserving. Thus if $T$ is infinite measure-preserving and not W-measurably sensitive for a compatible metric $d$, then 
  nonempty $d$-open sets have infinite measure.
\end{remark}

\section{Entropy and Sensitivity} 
\label{sec:entropy_and_sensitivity}


Sensitivity is only one way to capture the notion of chaos. It is generally accepted that entropy should be stronger than sensitivity notions. We show that positive entropy implies WM-sensitivity, improving a result of Cadre and Jacob \cite{Cadre05}. The preceding sections suggest that Li-Yorke M-sensitivity is a fundamentally stronger concept than W-measurable sensitivity.  Here we demonstrate examples of transformations which are W-measurable sensitive but not Li-Yorke M-sensitive.

\begin{proposition}\label{prop:2}Let $T$ on $(X,\mu)$ be an ergodic finite measure-preserving transformation. If $T$ has positive entropy, it is W-measurably sensitive.

\end{proposition}
\begin{proof}\label{pf11:}
We show the contrapositive.  If $T$ is not W-measurably sensitive then we know it is measurably isomorphic to a compact group rotation (\cite{G-S08}, Theorem 2). By the exact same argument as in the proof of Proposition \ref{prop:imp} this group is abelian. It is well know that compact abelian group rotations have zero entropy (\cite{Wa82}, Theorem 4.25).
\end{proof}

\begin{corollary}All finite measure-preserving ergodic transformations which have positive entropy but which are not weak mixing are W-measurably sensitive and not Li-Yorke M-sensitive. \end{corollary}

\begin{remark}
We now construct examples of W-measurably sensitive but not Li-Yorke M-sensitive transformations with zero entropy.  Let $(X,\mu,T)$ be a finite measure-preserving doubly ergodic transformation with zero entropy and let $(Y,\nu,R)$ be the rotation on two points.  Consider the dynamical system $(X\times Y,\mu\times\nu,T\times R)$.  This clearly has zero entropy.  We show this system is W-measurably sensitive and not Li-Yorke M-sensitive.  Let $S=T\times R$.  To see it is not Li-Yorke M-sensitive, one could note that it has a rotation factor and apply Theorem \ref{thm:weak_mixing}.  More explicitly, one could consider any $\mu$-compatible metric $d_X$ on $X$ and the $\mu$-compatible metric $d_Y$ on $Y$ which has $d_Y(1,2)=1$.  Then define $d$ on $X\times Y$ by 
\begin{equation*}
	d((x_1,y_1),(x_2,y_2)) = d_X(x_1,x_2) + d_Y(y_1,y_2).
\end{equation*} 
It is easy to see that $d$ is a $\mu$-compatible metric.  If $x_1,x_2 \in X$ then
\begin{equation*}
	d(S^n(x_1,1),S^n(x_2,2)) \geq 1
\end{equation*}
for all $n \geq 0$. Hence the pair $((x_1,1),(x_2,2))$ is not proximal and so not Li-Yorke. Thus the set $(X \times \{1\}) \times (X \times \{2\}) \subset (X \times Y)^2$ contains no Li-Yorke pairs and has positive measure, so $S$ is not Li-Yorke M-sensitive.

Now we show that the system is W-measurably sensitive. By Proposition 7.2 \cite{G-S08} it suffices to show that $S$ is measurably sensitive. Let $d$ be an arbitrary $\mu$-compatible metric on $X \times Y$. Then choose $\delta > 0$ such that $2 \delta < \diam(X \times \{i\})$. Choose a point $z = (x,i) \in X \times Y$ and consider $B_{\ve}(z)$. Then we know that there must exist an $i$ such that $\mu(B_{\epsilon}(z) \cap (X \times \{i\})) > 0$. Let $E$ be the projection of $B_{\ve}(z) \cap (X \times \{i\})$ onto $X$. Note that $E$ has positive measure. Now choose sets $A,B \subset X$ of positive measure such that $d(A,B) >2 \delta$. Since $T$ is doubly ergodic we can find an $n$ such that
\begin{align*}
	\mu(T^{-n}(A) \cap E) &> 0,\\
	\mu(T^{-n}(B) \cap E) &> 0.
\end{align*}
Then we see that $\diam(T^n(E)) >2 \delta$ so that by the triangle inequality there must exist a $z' \in E$ such that $d(T^n z, T^n z') > \delta$. Hence $T$ is measurably sensitive and thus W-measurably sensitive. 
\end{remark}

\section{Sensitive Metrics} 
\label{sec:sensitive_metrics}

We show   that under  mild hypotheses, every nonsingular transformation admits a $\mu$-compatible metric
that is W-measurably sensitive, and also  that a transformation where the set
of periodic points has positive measure never admits a $\mu$-compatible metric that is sensitive.
  This motivates the definition requiring that sensitivity be exhibited for all $\mu$-compatible metrics. We note that    the existence of a $\mu$-compatible metric for every
 transformation follows from  Proposition~\ref{prop:compatible_metric}.
After the research for this work was completed, we learned of the recent work of Morales \cite{Mo13}, where he obtains results related to Corollary 8.3 and Proposition 8.5 below, though
 the methods and context differ: where Morales investigates partitions,
we study metrics. Using \cite[Theorem 2.6]{Mo13}, it can be shown that the existence of a 
 sensitive metric implies the existence of a sensitive partition.  Measure expansive, being defined in terms of partitions, is then related but not equivalent to W-measurable sensitivity, as irrational rotations are measure expansive \cite[Example 2.2]{Mo13} but not W-measurably sensitive.  It would be interesting to investigate further connections.


We start with a generalization, suggested by the referee, of our original proof that assumed conservative ergodic; this needs a preliminary 
definition.
 We say a transformation $T$ has
{\it spillover} if there exists a countable collection of Borel sets $A_n$ such that
$A_n\subset A_{n+1}$, $\bigcup_{n=1}^\infty A_n=X$, and  $(A_n)_T$ has measure zero for each $n$, where
for a set $A$ we  write $A_T=\bigcap_{i=0}^\infty T^{-i}(A)$ (the largest positively-invariant set in $A$). Clearly, when $T$ is conservative ergodic then it satisfies spillover, and so does $T\times S$ for any nonsingular $S$.

\begin{proposition}\label{P:spillover}
Let  $(X,\mu,T)$ be nonsingular system that satisfies the spillover condition. 
If $T$ admits   a $\mu$-compatible metric $d'$ for which $T$ is an isometry, then 
$T$ admits   a $\mu$-compatible metric $d$ for which $T$ is W-measurably sensitive.
\end{proposition}

\begin{proof} 
We may assume $\mu$ is a probability measure and the sets $A_n$ have been chosen
so that $\mu(X\setminus A_n)<\frac{1}{2n}$. For each $n$ cover $X$ with a countable
collection of $d'$-balls of radius less than $1/2n$. Then there exist finitely many such balls,
denote them $A_{n,1}\ldots, A_{n,k_n}$, such that $\mu(X\setminus \bigcup_{i=1}^{k_n}A_{n,i})<\frac{1}{2n}$.   Write $A^\prime_{n,i}=A_n\cap A_{n,i}$. Then  the union $\bigcup_{n,i} A_{n,i}$ 
 has full measure; let its measure zero complement be $E_0$. Rename the sets
$A_{n,i}^\prime$ as $D'_0, D'_1, \ldots$ and then disjointify them by defining 
$D''_0=D'_1, D''_n= D'_n\setminus (D'_0\cup\ldots\cup D'_n)$. We may assume $D''_0$ has positive measure and is in $A_1$. While the metric $d'$ may not be $\mu$-compatible 
when restricted to each  $D''_n$ of positive measure,  it is separable, so it is $\mu$-compatible after removing a set of measure zero from each $D''_n$ of positive measure. Let $E_1$ the union of all the measure zero sets that have been removed, and all of the $D''_n$ of measure zero. Let $Z=\bigcup (A_n)_T$. Adjoin the measure zero sets $E_0, E_1, Z$ to $D''_0$ and rename it $D_0$. Rename the remaining 
positive measure sets $D''_n$ by $D_1, D_2, \ldots$. Now the sets $D_n$ form a positive measure partition of $X$, $d'$ is $\mu$-compatible when restricted to each $D_n, n>0$, and the $d'$ diameter of $D_n$ decreases to $0$. Choose  a $\mu$-compatible metric $d_0$ on  $D_0$.
When $x\in D_n$ write $N(x)=n$.  Then define a $\mu$-compatible metric $d$ on $X$ to be $d_i$ for two points in $D_i$ and one otherwise.
 
Note that $Z_0=\bigcup_{i\geq 0} T^{-i}(E_0\cup E_1)$ is of measure zero. Let  $x\notin Z\cup Z_0$. As $x\notin Z$,  the sequence  $N(T^i(x))$ is unbounded. Let $y\in X, y\neq x$
and set $\delta=d'(x,y)>0$. There exists $K$ so that for all
$n>K$ the $d'$-diameter of $D_n$ is less than $\delta$.  Since $d'(T^i(x),T^i(y))=d'(x,y)$,
when $N(T^i(x))>K$, then $T^i(x)$ and $T^i(y)$ cannot be in the same $D_n$
for $n>K$. Thus $d'(T^i(x),T^i(y))=1$ for infinitely many $i$. Therefore 
$T$ is W-measurably sensitive for $d$.
\end{proof}

Recall that a transformation is {\it aperiodic} if its set of periodic points has measure zero.
In the following lemma we need Rohlin's Lemma for nonsingular transformations that are not
necessarily invertible. We use a version proved in \cite[Theorem 2 and Remark 1]{AB06}: If $T$ is nonsingular and forward
measurable, for each $\varepsilon>0$ and each $n\in\N$, there exists a measurable set $C$
with $\mu(C)<1/n+\varepsilon$ and 
such that $C, T^{-1}(C),\ldots T^{-n+1}(C)$ are disjoint and of measure $>1-\varepsilon$.

\begin{lemma}\label{L:apspill}
Let $T$ be a nonsingular transformation and suppose it is forward measurable.  Then $T$ is aperiodic if and only if $T$ satisfies the spillover condition.
\end{lemma}

\begin{proof} Clearly, if the set of periodic points has positive measure, the transformation does not satisfy the spillover condition. Suppose, now, that $T$ is aperiodic. 
  Let $k\in\N$. By Rohlin's Lemma, for $\varepsilon=1/2^k$ and $n=2^k$ there exits
a set $C_k$ such that $C_k, T^{-1}(C_k),\ldots, T^{-n+1}(C_k)$ are disjoint with 
\[\mu(\bigsqcup_{i=0}^{n-1}T^{-i}(C_k))>1-\frac{1}{2^k}\text{ and }\mu(C_k)<\frac{2}{2^k}.\]
Let
\[
B_k=\bigsqcup_{i=2}^{n-1}T^{-i}(C_k).
\]
Then if $x\in B_k$  there is some $j$ so that $T^j(x)\in C_k$, which is disjoint from $B_k$, so
$\mu([B_k]_T)=0$. Also $\mu(B_k)>1-\frac{3}{2^k}$. Finally set 
\[A_m= \bigcap_{k\geq m}B_k.\]
Then one can verify that $A_m\subset A_{m+1}$ and $\mu([A_m]_T)=0.$ Furthermore,
\begin{align*}
\mu(A_m)&=\mu(\bigcap_{k\geq m}B_k)=\mu(X\setminus \bigsqcup_{k\geq m}(B_k)^c)\\
&>1- (\frac{3}{2^m}+\frac{3}{2^{m+1}}+\cdots)=1-\frac{3}{2^{m-1}}.
\end{align*}
Therefore $\bigcup_{m\geq 1}A_m=X$ mod $\mu$ and $T$ satisfies spillover.
\end{proof}

\begin{corollary}\label{C:existscompatible}
Let $(X,\mu,T)$ be a nonsingular system. 	If  $T$ is aperiodic, forward measurable and admits a $\mu$-compatible isometry, or if $T$ is  conservative ergodic, then there exists a $\mu$-compatible metric $d$ for which $T$ is W-measurably sensitive.
\end{corollary}

\begin{proof}The first part follows from Proposition~\ref{P:spillover} and Lemma~\ref{L:apspill}.
When $T$ is conservative ergodic, if it is W-measurably sensitive then we are done, so suppose it is not. Then by Theorem~\ref{prop:wmsclass},  there exists a $\mu$-compatible metric $d'$ on $X$ for which $T$ is an invertible isometry.  As $T$ must satisfy the spillover condition, Proposition~\ref{P:spillover} completes the proof.
\end{proof}

We have  shown the existence of a sensitive metric for a large class of transformations.  We now  prove that there exists a class of transformations   which  do not admit any $\mu$-compatible metric. We are indebted to the referee who suggested to us an argument simpler than our original proof.  We start we a lemma from  \cite{J-S08, G-S08} whose proof we include for completeness; a similar lemma can be found in \cite{ZP12}, which we learned of only recently.

\begin{lemma}\label{L:sepcompmetric} Let $(X,\mu)$ be a standard space with a nonatomic $\sigma$-finite metric.
If $d$ is a separable Borel metric on $X$, then there exists a measure zero set $Z$ such 
that $d$ is $\mu$-compatible on $X\setminus Z$. If $T$ is nonsingular on
$(X,\mu)$, then $Z$ can be chosen so that $X\setminus Z$ is positively invariant.
\end{lemma}

\begin{proof} Let $Z$ consist of the union of all measure zero $d$-open sets. As the metric is separable, it is Lindel\"of, and so is $Z$. Thus $Z$ is a countable union of null sets, so it is null.
When $T$ is nonsingular, let $Z_1=\bigcup_{n=0}^\infty T^{-n}(Z)$. So $Z_1$ is null and $X\setminus Z_1$ is positively invariant.
\end{proof}

\begin{proposition}\label{T:existscompatible}
Let $(X,\mu,T)$ be a nonsingular system. 	
If $T$ is not aperiodic, then there exists no $\mu$-compatible metric $d$ on $X$  for which $T$ is W-measurable sensitive.
\end{proposition}

\begin{proof}\label{pf5:}
\label{pf6:} First we show there exists a positively invariant set of positive measure $X_1$ 
and a positive integer $N$ such that $T^N$ restricted to $X_1$ is the identity. 
Let $Per$ denote the set of periodic points of $T$. For each $n\in\N$ write $P_n=\{x\in X: T^{n!}(x)=x\}$. Then each $P_n$ is positively invariant, $P_n\subset P_{n+1}$ and $Per=\bigcup_{n\geq 1}P_n$. 
As $\mu(Per)=0$, then some $P_n$ has  positive measure and  $T^{n!}$ is the identity on $P_n$.
Let   $X_1=P_n$ and  $N=n!$.

Suppose now there exists a $\mu$-compatible metric $d$ for which $T$ is W-measurably sensitive with sensitivity constant $\delta > 0$.  
While $d$ may not be $\mu$-compatible on $X_1$, it is separable. It follows that the metric $d_N$ defined by
\[d_N(x,y)=\sum_{i=0}^{N-1}d(T^ix,T^iy)\]
is separable. By Lemma~\ref{L:sepcompmetric}, after deleting a set of measure zero,
we may assume $d_N$ is $\mu$-compatible on a positively invariant subset of full measure of $X_1$, that we may rename $X_1$.  Let $x\in X_1$ and let $0<\ve<\delta/2$.
Then the $d_N$ ball $B=B^{d_N}(x,\ve)$ has positive measure. So  for all $y,z\in B$, 
\[
d(T^iy,T^iz)\leq d_N(T^iy,T^iz) = d_N(y,z)<2\ve<\delta.
\]
Therefore $T$ is not W-measurably sensitive for $d$, a contradiction. Thus  no sensitive metric exists.
\end{proof}

\bibliographystyle{plain}
\bibliography{LiYorkeRef}

\end{document}